\documentclass[11pt,twoside, reqno]{amsart}

\usepackage{amsmath}
\usepackage{amsthm}
\usepackage{amsfonts, amssymb}
\usepackage{mathrsfs}
\usepackage[all]{xy}
\usepackage{url}

\usepackage{latexsym}
\usepackage[dvips]{graphicx}

\theoremstyle{plain}
\newtheorem{lema}{Lemma}[section]
\newtheorem{prop}[lema]{Proposition}
\newtheorem{teo}[lema]{Theorem}

\newtheorem{coro}[lema]{Corollary}
\theoremstyle{remark}

\newtheorem{obs}[lema]{Remark}
\theoremstyle{definition}
\newtheorem{defi}[lema]{Definition}
\newtheorem{ej}[lema]{Example}

\newcommand{\Z}{\mathbb{Z}}
\newcommand{\N}{\mathbb{N}}
\newcommand{\R}{\mathbb{R}}
\newcommand{\T}{\mathbb{T}}

\newcommand{\PP}{\mathcal{P}}

\newcommand\restr{\raisebox{-0.3ex}{$|$}\raisebox{0.3ex}{}}
\newcommand{\tr}{\textrm{tr}}
\newcommand{\K}{\mathcal{K}}
\newcommand{\X}{\mathcal{X}}
\newcommand{\lee}{\le}
\newcommand{\susc}{\textrm{susc}}
\newcommand{\slsc}{\textrm{slsc}}
\def\spi#1{#1^{-1}}
\def\lpi#1{#1^{*-1}}
\def\opn#1{U_{#1}}
\def\cl#1{\textrm{cl}\, #1}

\begin{document}

\title[A Lefschetz fixed point theorem for multivalued maps of finite spaces]{A Lefschetz fixed point theorem for multivalued maps of finite spaces}

\author[J.A. Barmak]{Jonathan Ariel Barmak$^{\dagger}$}
\author[M. Mrozek]{Marian Mrozek$^{*}$}
\author[T. Wanner]{Thomas Wanner$^{\ddagger}$}

\thanks{$^{\dagger}$ Researcher of CONICET. Partially supported by grant UBACyT 20020160100081BA}
\thanks{$^*$ Partially supported by Polish NCN Ma\-estro Grant 2014/14/A/ST1/00453}
\thanks{$^{\ddagger}$ Partially supported by NSF grant DMS-1407087}

\address{Jonathan A. Barmak. Universidad de Buenos Aires. Facultad de Ciencias Exactas y Naturales. Departamento de Matem\'atica. CONICET-Universidad de Buenos Aires. Instituto de Investigaciones Matem\'aticas Luis A. Santal\'o (IMAS). Buenos Aires, Argentina. } \email{jbarmak@dm.uba.ar}

\address{Marian Mrozek. Division of Computational Mathematics, Faculty of Mathematics and Computer Science, Jagiellonian University, ul. St. Lojasiewicza 6, 30-348 Krak\'ow, Poland.} \email{Marian.Mrozek@uj.edu.pl}

\address{Thomas Wanner. Department of Mathematical Sciences, George Mason University, Fairfax, VA 22030, USA.} \email{twanner@gmu.edu}

\begin{abstract}
We prove a version of the Lefschetz fixed point theorem for multivalued maps $F:X\multimap X$ in which $X$ is a finite $T_0$ space.
\end{abstract}



\maketitle

\section{Introduction}

For over a century dynamical systems have been used to study time-evolving phenomena in the applied sciences. Based on the fundamental assumption that the future evolution of the system is completely determined by its initial state, dynamical systems can be divided into two broad categories. Continuous-time dynamical systems model the evolution of an initial state for all times $t \in \R$ or $t \in \R_0^+$, while discrete-time dynamical systems are only interested in the discrete times $t \in \Z$ or $t \in \N_0$. Once one of these four time sets~$\T$ is chosen, and if the topological space~$X$ denotes the underlying state space of the system, then a dynamical system is a continuous map $\varphi : \T \times X \to X$ which satisfies
\begin{displaymath}
  \varphi(0,x) = x
  \quad\mbox{ and }\quad
  \varphi(t+s,x) = \varphi(t,\varphi(s,x))
  \quad\mbox{ for all }\quad
  t,s \in \T \; , \;\; x \in X \; .
\end{displaymath}
If $x \in X$ denotes an initial state of the system, then $\varphi(t,x)$ denotes the uniquely determined state of the system at time $t \in \T$.

While the concept of dynamical systems is fairly abstract, they can easily be generated in applications. On the one hand, if we consider the state space $X = \R^d$, then under mild regularity assumptions on a vector field $f : \R^d \to \R^d$ solutions of the autonomous differential equation $\dot{x} = f(x)$ give rise to a continuous-time dynamical system. On the other hand, if~$X$ is any topological space and $f : X \to X$ is continuous, then we obtain a discrete-time dynamical system by letting $\varphi(k,x) = f^k(x)$. In other words, discrete-time dynamical systems correspond to iterations of a fixed map. Notice that one can always choose the discrete time set $\T = \N_0$, but that the choice $\T = \Z$ requires~$f$ to be a homeomorphism.

The primary focus of the theory of dynamical system is to describe the behavior of its orbits. For any given initial state $x \in X$, the orbit through~$x$ is the image of the map~$\varphi(\cdot,x)$. Both in the case of ordinary differential equations and the iteration of maps, one usually cannot derive explicit formulas for the state~$\varphi(t,x)$ for arbitrarily large times~$t$, and therefore the focus has shifted towards the development of a qualitative theory. Building on the properties of simple orbits such as equilibria or periodic orbits, qualitative theory aims to assemble a global picture of the dynamics of~$\varphi$, and over the last century an impressive body of work has been accumulated toward this goal, based to a large part on topological methods. See for example~\cite{Rob} and the references therein.

Yet, is has become increasingly clear that for dynamical systems which arise in concrete situations the application of these abstract mathematical results often poses practical challenges. For example, even for simple high-dimensional ordinary differential equations of the form $\dot{x} = f(x)$ one usually cannot determine all of its equilibrium solutions explicitly, since they are given by the solutions of the nonlinear system $f(x) = 0$. With the advent of powerful computational techniques, numerical computations have increasingly been used to analyze the behavior of such dynamical systems, but this usually does not lead to mathematically rigorous results. See for example the discussion in~\cite{SanWan}.

In order to overcome the challenges of concrete applications while still retaining mathematical rigor, a number of researchers have started to employ computer-assisted proof techniques to study the global dynamics of a system. One promising approach is based on discretization. Since the state space is usually an infinite Hausdorff space such as~$X = \R^d$, one can introduce a coarser finite representation of the space as the union $X = G_1 \cup \ldots \cup G_N$ of certain subsets, which are usually based on a grid on~$X$. Induced by the underlying dynamical system~$\varphi$ one can then try to determine the collective behavior of all initial states in a set~$G_k$. Which sets can they move into? In the case of continuous-time dynamical systems, states originating in~$G_k$ can move to a subset of the neighboring cells, while in the case of a discrete-time dynamical system based on the map~$f : X \to X$, states from~$G_k$ can be mapped into the cells which intersect the image~$f(G_k)$. Moreover, even if the dynamical system is not known explicitly, one could still determine all images of the cell~$G_k$ through approximations which include computable error bounds and lead to potentially larger target sets. In either case, this discretization process will lead to a multivalued map, which in some sense approximates the dynamical behavior. Despite the large loss of information inherent in this discretization, one can often still use topological results from degree theory~\cite{Zei} or the more dynamically oriented Conley index theory~\cite{Con} to transfer statements about the finite multivalued map back to mathematical results about the underlying dynamical system. For example, in~\cite{MisMro} this approach was used to prove chaos in the Lorenz equations.

While the above-described procedure naturally leads to the study of multivalued maps on a topological space, such maps have been studied extensively before in a number of other contexts, such as for example control theory and differential inclusions~\cite{AubCel}. These applications have shown that one needs to be careful with the formulation of smoothness assumptions for multivalued maps, as there is no canonical notion of continuity. In order to keep the theory as general as possible, two notions of semicontinuity have been introduced, and these lead to a theory of multivalued maps which in many aspects parallels the treatment of their single-valued counterparts. For example, topological methods to establish the existence of fixed points are developed in~\cite{Gor} under the assumption of semicontinuity. The majority of these results, however, has only been developed for multivalued maps on Hausdorff spaces.

To the best of our knowledge, there is no comprehensive theory aimed at the dynamics of multivalued maps between finite topological spaces --- which are Hausdorff only under the trivial discrete topology. In addition to the computer-assisted proof approach outlined earlier,  multivalued maps on finite topological spaces constitute a natural tool to study sampled dynamical systems, that is systems known only from a finite sample, for example a time series \cite{ABMS2015,DJKKLM2018}. Another example is Forman's theory of combinatorial vector fields~\cite{For1, For2}. The theory introduces an associated flow on the underlying simplicial complex which can be viewed as a multivalued map. The dynamics of this flow map has been studied in a recent series of papers~\cite{BKMW, KMW, Mro}. More precisely, the authors introduce combinatorial counterparts to the dynamical concepts of isolated invariant sets, Morse decompositions, and Conley index in the setting of simplicial complexes and Lefschetz complexes. In addition, they establish connections between the combinatorial theory and its classical versions. The generalization of this work to general finite topological spaces, however, remains open.

As a first step towards such a theory, the present paper is devoted to deriving a Lefschetz fixed point theorem for multivalued maps $F:X\multimap X$ where $X$ is a finite $T_0$ space. Recall that there is an isomorphism between the category of finite $T_0$ spaces with continuous maps and the category of finite posets with order-preserving maps. Multivalued maps between posets were previously investigated in \cite{RS, Smi, Smi2, Wal2}. Most of these results were aimed at studying fixed points in the case of infinite posets. In particular in \cite{RS} and \cite{Wal2}, multivalued maps are used to study the problem of the fixed point property of a product of posets. The maps considered in these articles are upper semicontinuous (usc) and/or lower semicontinuous (lsc), and \cite{Wal2} used the notion of isotone maps for maps which are both usc and lsc. In the classical fixed point theory of multivalued mappings, fixed points of usc multivalued maps with acyclic values $F:X\multimap X$ between absolute neighborhood retracts are studied using a variant of the Lefschetz fixed point theorem \cite[Theorem~32.9]{Gor}. We will see that there is no hope to define Lefschetz numbers for multivalued maps $F:X\multimap X$ between finite $T_0$ spaces which are isotone and have acyclic values if we want each continuous selector $f:X\to X$ of $F$ to have the same Lefschetz number as $F$ (see Proposition \ref{noway}). Rather, we concentrate on multivalued maps which satisfy stronger regularity properties than upper or lower semicontinuity, which will be called strong upper and strong lower semicontinuity. For such multivalued maps it will be possible to define an induced map in homology, and then to establish the Lefschetz fixed point theorem.

The remainder of this paper is organized as follows. After a brief review of finite spaces in Section~\ref{secfs}, we recall the notions of upper and lower semicontinuity for multivalued maps and introduce stronger versions of these concepts in Section~\ref{secsc}. In addition, we provide a number of equivalent characterizations for these definitions in the context of finite $T_0$ spaces, which will be useful later on. In Section~\ref{sechomhom} we show that multivalued maps which are strongly usc or strongly lsc and which have acyclic values induce well-defined maps in homology, and this is used to establish the Lefschetz fixed point theorem in Section~\ref{secfixedpointthm}. Section~\ref{secfpp} discusses the fixed point property for finite $T_0$ spaces with respect to these multivalued maps, while Section~\ref{secisotone} shows that it is not possible to define a Lefschetz number for isotone maps. Finally, in Section~\ref{sechomotopy} we present a notion of homotopy which preserves Lefschetz numbers.

\section{Basics on finite spaces} \label{secfs}

We begin by recalling the basic correspondence between finite spaces and posets due to Alexandroff and some elementary results of the homotopy theory of finite spaces originally developed by McCord and Stong. A finite topological space is a space with finitely many points. Many of the results of this paper can be stated for arbitrary finite spaces but we will restrict ourselves to finite $T_0$ spaces, that is finite spaces in which for any two different points there exists an open set containing only one of them. If~$X$ is a finite $T_0$ space, then for every point $x \in X$ there exists a {\em smallest open set\/}~$\opn{x}$ which contains~$x$. If we then define
\begin{displaymath}
  x \le y
  \quad \textrm{if} \quad
  x \in \opn{y} \; ,
\end{displaymath}
then~$X$ becomes a poset with respect to the so-defined order. One can easily show that we also have
\begin{displaymath}
  x \ge y
  \quad\Leftrightarrow\quad
  x \in \cl{y} \; ,
\end{displaymath}
where as usual~$\cl{y}$ denotes the closure of the set~$\{ y \}$.
Conversely, if~$X$ denotes any finite poset, then one obtains a finite $T_0$ space by considering all down-sets in the poset as open. Recall that $U\subset X$ is called a down-set if $x\in U$ and $y\le x$ implies $y\in U$. The closed sets of this space are then the up-sets.\footnote{We would like to point out that while this specific correspondence seems to be the most commonly used one, it would also be possible to define the order relation by replacing $\opn{y}$ by $\cl{y}$. This was in fact done in the paper \cite{Ale} which first established the connection between posets and finite $T_0$ spaces. In this convention, down-sets in the poset correspond to closed sets, and up-sets to open sets. All of the results in the present paper remain valid under this convention, if one reverses all poset inequalities.} This establishes a correspondence between finite $T_0$ spaces and finite posets and from now on we will use this correspondence to treat finite $T_0$ spaces and finite posets as the same object.

If $A$ is any subset of a finite $T_0$ space $X$, we denote by $U_A$ the smallest open set containing $A$. This is exactly the set of points in $X$ which are smaller than or equal to some point in $A$.

For finite $T_0$ spaces $X$ and $Y$ the product topology on $X\times Y$ corresponds to the product order given by $(x,y)\le (x',y')$ if $x\le x'$ and $y\le y'$.

It is easy to prove that a single-valued function $f:X\to Y$ between finite $T_0$ spaces is continuous if and only if it is order-preserving, that is $x\le x'$ implies $f(x)\le f(x')$. Given a finite $T_0$ space $X$, we denote by $X^{op}$ the finite space with the dual (opposite) order. Therefore a map $f:X\to Y$ between finite $T_0$ spaces is continuous if and only if the map $f^{op}:X^{op}\to Y^{op}$ which coincides with $f$ in the underlying sets is continuous.

If $X$ is a finite $T_0$ space, two points $x,y\in X$ lie in the same path-component if and only if there exists a sequence $x=x_0\le x_1 \ge x_2\le \ldots x_n=y$. Such a sequence is called a fence from $x$ to $y$. Given finite $T_0$ spaces $X$ and $Y$, the set $Y^X$ of continuous maps from $X$ to $Y$ has a natural order, the \textit{pointwise order}, given by $f\le g$ if $f(x)\le g(x)$ for every $x\in X$.
If $f\le g$ then the map $H:X\times [0,1]\to Y$ defined by $H(x,0)=g(x)$ and $H(x,t)=f(x)$ for $t>0$ defines a homotopy from $g$ to $f$. In particular, a finite $T_0$ space with a maximum is contractible. The topology which corresponds to the pointwise order is the compact-open topology. This fact was used by Stong \cite{Sto} to give a characterization of homotopies between maps of finite spaces: two maps $f,g:X\to Y$ are homotopic if and only if there is a sequence $f=f_0\le f_1\ge f_2\le \ldots f_n=g$ of continuous maps from $X$ to $Y$.

In general finite spaces are not homotopy equivalent to Hausdorff spaces. However, McCord proved that any finite space is weak homotopy equivalent to a polyhedron. Recall that a map $f:X\to Y$ between arbitrary topological spaces is said to be a weak homotopy equivalence if it induces isomorphisms in all homotopy groups for any base point. One such map induces automatically isomorphisms in all the homology groups for any coefficient group.

\begin{teo} (McCord's Theorem \cite[Theorem 6]{Mcc}) \label{mccord}
Suppose $X$ is any space and $Y$ is a finite $T_0$ space. Let $f:X\to Y$ be a continuous map such that $f^{-1}(U_y)$ is weakly contractible (i.e., it has trivial homotopy groups) for every $y\in Y$. Then $f$ is a weak homotopy equivalence.
\end{teo}

In the case that $X$ is finite as well this result can be deduced from Quillen's Theorem A \cite{Qui2}. If $X$ is a finite $T_0$ space, we denote by $\K (X)$ the order complex of $X$, that is, the simplicial complex consisting of all chains in the poset. McCord used the result above to prove that for any finite $T_0$ space there exists a weak homotopy equivalence $\mu_X:\K(X) \to X$ \cite[Theorem 1]{Mcc}. Moreover, any continuous map $f:X\to Y$ between finite~$T_0$ spaces induces a simplicial map $\K(f):\K(X) \to \K(Y)$ which coincides with~$f$ on vertices, and we have $f \circ \mu_X = \mu_Y \circ \K(f)$. Note that since $\K(X)=\K(X^{op})$, $X$ and $X^{op}$ have isomorphic homology groups with an isomorphism given by $(\mu_{X^{op}})_*(\mu_X)_*^{-1}:H_n(X)\to H_n(X^{op})$.

In the classical setting of compact spaces the construction of homology of multivalued maps  is  based on Vietoris-Begle mapping theorem \cite{Vie}.
In the setting of finite topological spaces this theorem may be replaced by
the following version of McCord's result for homology. Recall that we call a space acyclic if all its reduced homology groups with integer coefficients are trivial.

% whose proof can be found in \cite[p.~116]{Qui}, \cite[Theorem~2.3]{Wal}, \cite[Corollary~4.3]{BWW}, or \cite[Corollary~6.5]{Bar}. 

\begin{teo} [McCord's Theorem for homology] \label{quillenhomology}
Let $X$ be an arbitrary space and let $Y$ be a finite $T_0$ space. Let $f:X\to Y$ be a continuous map such that $f^{-1}(U_y)$ is acyclic for every $y\in Y$. Then $f$ induces isomorphisms $f_*:H_n(X)\to H_n(Y)$ in all the homology groups with integer coefficients.
\end{teo}
\begin{proof}
We proceed by induction on the cardinality of $Y$. The result is true when $Y$ is empty. For $Y$ non-empty let $y\in Y$ be a maximal point. Then $f\restr_{f^{-1}(Y\smallsetminus \{y\})}: f^{-1}(Y\smallsetminus \{y\}) \to Y\smallsetminus \{y\}$ and $f\restr_{f^{-1}(U_y\smallsetminus \{y\})}: f^{-1}(U_y\smallsetminus \{y\})\to U_y\smallsetminus \{y\}$ induce isomorphisms in homology by induction. Since $U_y$ has a maximum, it is contractible so by hypothesis $f\restr _{f^{-1}(U_y)}:f^{-1}(U_y)\to U_y$ induces isomorphisms in homology as well. A Mayer-Vietoris argument together with the five lemma shows that $f:X\to Y$ induces isomorphisms in homology.
\end{proof}

Other versions of this result have been studied in \cite{Bar, BWW, Qui, Wal}.

%We provide now similar characterizations of the notions of semicontinuity of multivalued maps between finite spaces. If $X$ is a finite $T_0$ space and $A\subseteq X$ is a subset, we denote by $\opn{A}$ the smallest open set containing $A$. This is exactly the set of points $x\in X$ which are smaller than or equal to some element of $A$.

\section{Semicontinuity of multivalued maps}
\label{secsc}

We begin by recalling two fundamental regularity assumptions which can be imposed on multivalued maps. For this, let~$X$ and~$Y$ denote topological spaces, and let $F:X\multimap Y$ be an arbitrary multivalued map, that is, a map which associates a subset $F(x)\subset Y$ with every $x\in X$. In view of the standard definition of continuity for single-valued maps, one would like to have a notion of continuity which is based on the condition that inverse images of open sets are again open. While in the single-valued map case this leads to a well-defined notion, the meaning of inverse image in the case of multivalued maps is not immediately clear. In fact, one could use either the definition
\begin{displaymath}
  \spi{F}(B) = \left\{ x \in X \; : \;
    F(x) \subset B \right\}
  \quad\mbox{ or }\quad
  \lpi{F}(B) = \left\{ x \in X \; : \;
    F(x) \cap B \neq \emptyset \right\}
\end{displaymath}
for all subsets $B \subset Y$. We refer to these two definitions as the {\em small preimage\/} and the {\em large preimage\/} of~$B$ under~$F$, respectively, since clearly the first is contained in the second for multivalued maps with nonempty values. Depending on which notion is used, one then obtains the following two continuity concepts.

\begin{defi}[Semicontinuity]
Let $X$ and $Y$ denote two topological spaces, and let $F:X\multimap Y$ denote an arbitrary multivalued map between them. Then we say that~$F$ is {\em upper semicontinuous (usc)\/} if for every open set $B \subset Y$ the small preimage $\spi{F}(B)$ is open in~$X$. The multivalued map~$F$ is called {\em lower semicontinuous (lsc)\/} if for every open set $B \subset Y$ the large preimage $\lpi{F}(B)$ is open in~$X$.
\end{defi}

Notice that any continuous single-valued map $f : X \to Y$ can be viewed as a multivalued map $x\mapsto \{f(x)\}$, and this induced map is always usc and lsc. Thus, both of the above definitions are natural generalizations of the continuity concept to multivalued maps, but one can easily see that they are satisfied by different classes of maps. As was shown in~\cite{Gor}, depending on the specific application one or the other concept might be more appropriate. Notice also that for usc and lsc the closedness properties of preimages of closed sets are more delicate. Since one can easily show that
\begin{displaymath}
  X \setminus \lpi{F}(B) = \spi{F}( Y \setminus B )
  \quad\mbox{ for all }\quad
  B \subset Y \; ,
\end{displaymath}
we have the characterizations
\begin{eqnarray*}
  F \;\mbox{ is usc } & \Leftrightarrow &
    \lpi{F}(C) \;\mbox{ is closed for all closed sets }\;
    C \subset Y \; , \\
  F \;\mbox{ is lsc } & \Leftrightarrow &
    \spi{F}(C) \;\mbox{ is closed for all closed sets }\;
    C \subset Y \; ,
\end{eqnarray*}
in which the large and small preimages are switched.

\begin{lema}[Semicontinuity in finite $T_0$ spaces]
\label{lem:scft0s}
Let $X$ and $Y$ denote two finite $T_0$ spaces, and let $F:X\multimap Y$ denote an arbitrary multivalued map between them. Then the following four statements are pairwise equivalent.
\begin{itemize}
\item[($u_a$)] The map~$F$ is upper semicontinuous.
\item[($u_b$)] For every $x \in X$ we have $F(\opn{x}) \subset \opn{F(x)}$.
\item[($u_c$)] For all $x_1, x_2 \in X$ with $x_1 \le x_2$ we have $F(x_1) \subset \opn{F(x_2)}$.
\item[($u_d$)] For all $x_1, x_2 \in X$ with $x_1 \le x_2$ and for all $y_1 \in F(x_1)$ there exists a $y_2 \in F(x_2)$ such that $y_1 \le y_2$.
\end{itemize}
For the concept of lower semicontinuity, we obtain the following four pairwise equivalent statements.
\begin{itemize}
\item[($\ell_a$)] The map~$F$ is lower semicontinuous.
\item[($\ell_b$)] For every $x \in X$ we have $F(\cl{x}) \subset \cl{F(x)}$.
\item[($\ell_c$)] For all $x_1, x_2 \in X$ with $x_1 \ge x_2$ we have $F(x_1) \subset \cl{F(x_2)}$.
\item[($\ell_d$)] For all $x_1, x_2 \in X$ with $x_1 \ge x_2$ and for all $y_1 \in F(x_1)$ there exists a $y_2 \in F(x_2)$ such that $y_1 \ge y_2$.
\end{itemize}
\end{lema}
\begin{proof}
We only prove the equivalences for the case of upper semicontinuity, since lower semicontinuity can be treated analogously. See also the discussion at the end of this section.

($u_a$) $\Rightarrow$ ($u_b$): The set $\opn{F(x)}$ is open by definition, so $F^{-1}(\opn{F(x)})\subset X$ is open. Since $x\in F^{-1}(\opn{F(x)})$, the smallest open set containing $x$, $U_x$, is contained in $F^{-1}(\opn{F(x)})$.

($u_b$) $\Rightarrow$ ($u_c$): Due to $x_1 \le x_2$ we have $x_1 \in \opn{x_2}$. This gives $F(x_1) \subset F(\opn{x_2}) \subset \opn{F(x_2)}$, according to~($u_b$).

($u_c$) $\Rightarrow$ ($u_d$): This follows from the fact that $\opn{F(x_2)}\subset Y$ is the set of elements $y\in Y$ which are smaller than or equal to some element in $F(x_2)$.

($u_d$) $\Rightarrow$ ($u_a$): Suppose that $B \subset Y$ is open. We need to show that~$\spi{F}(B) \subset X$ is open or, equivalently, a down-set. For this, let $x_2 \in \spi{F}(B)$ be arbitrary, which implies $F(x_2) \subset B$, and consider a point $x_1 \le x_2$. According to~($u_d$), for every $y_1 \in F(x_1)$ there exists $y_2 \in F(x_2) \subset B$ such that $y_1 \le y_2$. Since $B$ is a down-set, $y_1\in B$. This proves that $F(x_1)\subset B$, so $x_1 \in \spi{F}(B)$. Thus, $\spi{F}(B)$ is a down-set.
\end{proof}

These characterizations equip us with a variety of ways for establishing upper or lower semicontinuity in the case of finite $T_0$ spaces. In fact, in the above-mentioned work~\cite{RS} upper and lower semicontinuity were defined via properties~($u_d$) and~($\ell_d$), respectively.

Lemma~\ref{lem:scft0s} illustrates in a remarkable way the inherent symmetry between the concepts of upper and lower semicontinuity in finite $T_0$ spaces. It also shows that these concepts explicitly depend on the topologies of both spaces~$X$ and~$Y$. This is no longer the case for the following two concepts, which are of central importance for the present paper.

\begin{defi}[Strong Semicontinuity]
\label{def:ssc}
Let $X$ be a topological space, let $Y$ be a set, and let $F:X\multimap Y$ denote an arbitrary multivalued map between them. Then~$F$ is {\em strongly upper semicontinuous ($\susc$)\/} if for every subset $B \subset Y$ the small preimage $\spi{F}(B)$ is open in~$X$. The multivalued map~$F$ is called {\em strongly lower semicontinuous ($\slsc$)\/} if for every subset $B \subset Y$ the large preimage $\lpi{F}(B)$ is open in~$X$.
\end{defi}

The above definition immediately shows that
\begin{eqnarray*}
  \mbox{$F$ is $\susc$} & \Leftrightarrow &
    \mbox{$F$ is usc with respect to the discrete topology on~$Y$} \\
  & \Leftrightarrow &
    \mbox{$F$ is usc with respect to {\em any\/} topology on~$Y$}
    \; ,
\end{eqnarray*}
and these equivalences remain valid if $\susc$ and usc are replaced by $\slsc$ and lsc, respectively. For the case of $X$ being a finite $T_0$ space, strong semicontinuity has a convenient characterization through a set-theoretic monotonicity condition.

\begin{lema}[Combinatorial characterization of strong semicontinuity]
\label{lem:sscft0s}
Let $X$ be a finite $T_0$ space, let $Y$ be any set, and let $F:X\multimap Y$ denote an arbitrary multivalued map between them. Then we have:
\begin{eqnarray*}
  \mbox{$F$ is $\susc$} & \Leftrightarrow &
    \mbox{for all }\;
    x_1, x_2 \in X
    \;\mbox{ with }\;
    x_1 \le x_2
    \;\mbox{ we have }\;
    F(x_1) \subset F(x_2) \; , \\
  \mbox{$F$ is $\slsc$} & \Leftrightarrow &
    \mbox{for all }\;
    x_1, x_2 \in X
    \;\mbox{ with }\;
    x_1 \le x_2
    \;\mbox{ we have }\;
    F(x_1) \supset F(x_2) \; .
\end{eqnarray*}
\end{lema}
\begin{proof}
We only establish the first equivalence, since the second one can be proved analogously. Suppose that~$F$ is $\susc$ and that $x_1\le x_2\in X$. Then $F^{-1}(F(x_2))\subseteq X$ is open and since $x_2\in F^{-1}(F(x_2))$, then $x_1\in \opn{x_2} \subset F^{-1}(F(x_2))$. That is, $F(x_1)\subset F(x_2)$.

Conversely, suppose that for all $x_1, x_2 \in X$ with $x_1 \le x_2$ we have $F(x_1) \subset F(x_2)$ and let $B\subset Y$. Let $x_2\in F^{-1}(B)$. If $x_1\le x_2$, then $F(x_1)\subset F(x_2)\subset B$. This shows that $\opn{x_2}\subseteq F^{-1}(B)$. Thus, $F^{-1}(B)$ is open.
\end{proof}

As we will see later on, the notion of strong semicontinuity will be used to show that a multivalued map between finite $T_0$ spaces induces a well-defined map in homology. While at first glance this seems to be a severe restriction, there are situations in which strong semicontinuity is equivalent to semicontinuity. To describe this, recall that a multivalued map~$F:X\multimap Y$ has closed (or open) values, if for all $x \in X$ the set $F(x) \subset Y$ is closed (or open) in the topological space~$Y$. Then the following result is immediate.

\begin{lema}[Equivalence of strong and regular semicontinuity]
\label{lem:escft0s}
Let $X$ and $Y$ denote two topological spaces, and let $F:X\multimap Y$ denote an arbitrary multivalued map between them. Then we have:
\begin{eqnarray*}
  \mbox{$F$ is usc with open values} & \Rightarrow &
    \mbox{$F$ is $\susc$} \; , \\
  \mbox{$F$ is lsc with closed values} & \Rightarrow &
    \mbox{$F$ is $\slsc$} \; .
\end{eqnarray*}
In other words, for multivalued maps with open values upper semicontinuity is equivalent to strong upper semicontinuity, and for multivalued maps with closed values lower semicontinuity is equivalent to strong lower semicontinuity.
\end{lema}
\begin{proof}
Suppose $F$ is usc with open values. Let $B\subset Y$ be any subset. We want to show that $F^{-1}(B)$ is open. Let $x\in F^{-1}(B)$. Then, $F(x)\subset B$. By hypothesis $F^{-1}(F(x))$ is an open neighborhood of $x$ contained in $F^{-1}(B)$. Therefore, $F^{-1}(B)$ is open. The second implication follows similarly.
\end{proof}

We close this section by observing that the notion of strong upper semicontinuity is completely natural when working with finite spaces. Indeed, a multivalued map~$F:X\multimap Y$ between finite $T_0$ spaces can be identified with a single-valued map $f_F : X \to \mathcal{P}(Y)$, where~$\mathcal{P}(Y)$ denotes the power set of~$Y$. The set~$\mathcal{P}(Y)$ has a natural poset structure given by inclusion, so it is a finite $T_0$ space. By Lemma \ref{lem:sscft0s}, $F$ is susc if and only if $f_F$ is order-preserving, i.e., continuous. Thus, strong upper semicontinuity of
the multivalued map $F$ corresponds to the continuity of the associated single-valued map $f_F$. Therefore, it could be natural to call the strong upper semicontinuous multivalued maps simply continuous. We do not do that for two reasons. Firstly, in the context of dynamics we need to be able to iterate $F$. To do that we need the domain and codomain of $F$ to be the same topological space.
This is not true in the case of $f_F$ even if $X=Y$.
Secondly, in the classical theory of multivalued maps, the term \textit{continuous} is already used for a map $F:X\multimap Y$ which is simultaneously upper semicontinuous and lower semicontinuous. One can show that a multivalued map is continuous
in this sense if and only if the single-valued map $f_F:X\to \mathcal{P}(Y)$ is continuous  with respect to a different topology
on $\mathcal{P}(Y)$: the so called ``finite topology'' \cite[Definition~1.7 and Corollary~9.3]{Mic}.

If $Y$ is a finite set, $\mathcal{P}(Y)^{op}$ denotes the power set of $Y$ with the order given by reverse inclusion. If~$X$ is a finite $T_0$ space, then a multivalued map $F:X\multimap Y$ is $\slsc$ if and only if the single-valued map $f_F : X \to \mathcal{P}(Y)^{op}$ is continuous.

Finally, we can relate strong upper to strong lower semicontinuity. If $X$ is a finite $T_0$ space, according to Lemma~\ref{lem:sscft0s} a multivalued map $F:X\multimap Y$ is $\susc$ if and only if the map $F':X^{op}\multimap Y$ which coincides with $F$ in underlying sets is $\slsc$. For the special case $Y = X$ which is mainly considered in the present paper, we have that $F:X\multimap X$ is $\slsc$ if and only if the map $F^{op}:X^{op}\multimap X^{op}$ which coincides with $F$ in the underlying sets is $\susc$. Even more is true. One can easily see that $F:X\multimap Y$ is lsc if and only if $F^{op}:X^{op}\multimap Y^{op}$ is usc. This provides us with another way to explain the symmetry in the statements in Lemma~\ref{lem:scft0s}.

\begin{obs} \label{susciffc}
If $X$ is a not necessarily finite poset, then it can also be seen as a topological space in which open sets are the down-sets. The sets $\opn{x}=\{y\in X \ y\le x\}$ constitute a basis for the topology. As before, a multivalued map~$F:X\multimap Y$ between any topological spaces~$X$ and~$Y$, can be identified with a single-valued map $f_F : X \to \mathcal{P}(Y)$, where~$\mathcal{P}(Y)$ is the space associated to the order given by the inclusion. Now, if $B\subset Y$ is any subset $F^{-1}(B)=\{x\in X | F(x) \subset B\}=f_F^{-1}(\opn{B})$. Therefore $F$ is susc if and only if $f_F$ is continuous. This generalizes the comment made above for $X,Y$ being finite.
\end{obs}

\section{Homomorphisms induced in homology}
\label{sechomhom}

The natural first step towards deriving a Lefschetz fixed point theorem for multivalued maps $F: X\multimap X$ between finite $T_0$ spaces is the definition of a Lefschetz number for such maps. This in turn requires that~$F$ induces a well-defined map in homology. While we restrict our attention to strongly semicontinuous maps as introduced in Definition~\ref{def:ssc} of the last section, we need one additional assumption. Recall that a multivalued map~$F$ is said to have acyclic values if for every $x \in X$, $F(x)$ is acyclic. We denote by $H_n(X)$ the unreduced homology groups of $X$ with integer coefficients.

We will need the following auxiliary result, whose construction is inspired by the case of upper semicontinuous multivalued maps with acyclic values discussed in~\cite[p.~160]{Gor}.

\begin{lema}
\label{lem:inducedhommap}
Let~$X$ be a finite $T_0$ space and let~$Y$ be any topological space. Let $F:X\multimap Y$ be a $\susc$ multivalued map with acyclic values. Consider~$F$ as the subspace of~$X \times Y$ which consists of all pairs~$(x,y)$ for which $y \in F(x)$. Let~$p_1 : F \to X$ denote the projection onto the first coordinate. Then~$(p_1)_*:H_n(F) \to H_n(X)$ is an isomorphisms for every $n\ge 0$.
\end{lema}
\begin{proof}
%We only consider the case that~$F$ is $\susc$. According to the discussion at the end of the last section, a multivalued map $F : X \multimap Y$ is $\slsc$ if and only if $F : X^{op} \multimap Y$ is $\susc$. Moreover, the finite~$T_0$ spaces~$X$ and~$X^{op}$ are weak homotopy equivalent since they have the same order complexes $\K(X) = \K(X^{op})$. Thus, the result for $\slsc$ maps can be directly obtained from the result for $\susc$ maps.

By McCord's Theorem for homology (Theorem \ref{quillenhomology}) it suffices to show that~$p_1^{-1}(U_x)$ $\subset X \times Y$ is acyclic for every $x\in X$. This will be accomplished by verifying that this space is homotopy equivalent to the image~$F(x) \subset Y$.

Define $i : F(x) \to p_1^{-1}(U_x)$ by $i(y) = (x,y)$. This map is well-defined and continuous. In addition, define $r : p_1^{-1}(U_x) \to F(x)$ by $r(z,y) = y$ for all $z \in U_x$ and $y \in F(z)$. Note that $z \le x$ implies $F(z) \subseteq F(x)$ due to Lemma~\ref{lem:sscft0s}, and therefore the map~$r$ is well-defined and continuous. These definitions readily show that we have both $ri = 1_{F(x)}$ and $ir(z,y) = (x,y)$ for all $z \in U_x$ and $y \in F(z)$, and in order to complete the proof we only need to show that~$ir$ is homotopic to the identity~$1_{p_1^{-1}(U_x)}$.

As explained in Section \ref{secfs}, there exists a homotopy $H:U_x \times [0,1] \to U_x$ from the constant function~$c_x$ to the identity~$1_{U_x}$ defined by $H(z,0)=x$ and $H(z,t)=z$ for all $t > 0$ and $z \in U_x$. Then $H' = H \times 1_Y : U_x \times Y \times [0,1] \to U_x \times Y$ is a homotopy from $c_x \times 1_Y$ to $1_{U_x\times Y}$, and its restriction to $p_1^{-1}(U_x) \times [0,1]$ gives the desired homotopy from the composition~$ir$ to~$1_{p_1^{-1}(U_x)}$. Thus, the inverse image~$p_1^{-1}(U_x)$ is indeed acyclic for every $x \in X$ and $(p_1)_* : H_n(F) \to H_n(X)$ is an isomorphism for each $n \ge 0$.
\end{proof}

\begin{defi}
Let~$X$ be a finite $T_0$ space and let~$Y$ be any topological space. If $F:X\multimap Y$ is a $\susc$ multivalued map with acyclic values, we define $F_*:H_n(X)\to H_n(Y)$ as the composition $F_*=(p_2)_* (p_1)_*^{-1}$, where $p_2:F\to Y$ denotes the projection onto the second coordinate.

If $F:X\multimap Y$ is a $\slsc$ multivalued map with acyclic values, it induces homomorphisms in homology in the following way. The map $F':X^{op}\multimap Y$ which coincides with $F$ in the underlying sets is $\susc$ and $F'_*:H_n(X^{op})\to H_n(Y)$ is already defined. Since $\K(X^{op})=\K(X)$, for each $n\ge 0$ there is a well-defined homomorphism $F_*=F'_*(\mu_{X^{op}})_*(\mu_X)_*^{-1}: H_n(X)\to H_n(Y)$.  

\end{defi}

Lemma \ref{lem:inducedhommap} holds also for $\slsc$ maps when $X$ and $Y$ are both finite $T_0$ spaces. Concretely, if $F:X\multimap Y$ is a $\slsc$ multivalued map with acyclic values, then $p_1:F\to X$ induces isomorphisms in homology. This can be proved from Lemma \ref{lem:inducedhommap} using that $F^{op}:X^{op}\multimap Y^{op}$ is $\susc$. The projection $p_1^{op}:F^{op}\to X^{op}$ induces isomorphisms in homology and then so does $p_1$. In this case, it can be proved that $F_*:H_n(X)\to H_n(Y)$ coincides with $(p_2)_*(p_1)^{-1}_*$. Therefore, for $X$ and $Y$ finite, $F_*$ can be defined as the composition $(p_2)_*(p_1)^{-1}_*$ in both cases, for $\susc$ and $\slsc$ maps. This result will not be needed in the present article, but we include a proof at the end of this section for future reference.

Note that if in the hypothesis of Lemma \ref{lem:inducedhommap} we ask for the values~$F(x)$ of~$F$ to be contractible subspaces of~$Y$ or, more generally, to be weakly contractible, then  Theorem~\ref{mccord} (McCord's Theorem) can be used to define $F_* : \pi_n(X) \to \pi_n(Y)$. Recall that an analogue of the Whitehead theorem does not hold for finite spaces and moreover there exist weakly contractible finite spaces which are not contractible~\cite[Example~4.2.1]{Bar2}.

Our definition of the Lefschetz number is analogous to the classical case.

\begin{defi}[Lefschetz number]
\label{def:lefschetz}
Let~$X$ denote a finite~$T_0$ space, and let $F : X \multimap X$ be a multivalued map with acyclic values which is $\susc$ or $\slsc$.
Let~$F_* : H_n(X) \to H_n(X)$ denote the induced maps in homology, and for every $n \ge 0$ let~$T_n(X)$ be the torsion subgroup of~$H_n(X)$. Then the {\em Lefschetz number\/}~$L(F)$ of~$F$ is defined as
\begin{displaymath}
  L(F) = \sum_{n = 0}^\infty (-1)^n \,\tr (F_n) \; ,
\end{displaymath}
where $\tr (F_n)$ is the trace of the homomorphism $H_n(X) / T_n(X) \to H_n(X) / T_n(X)$ induced by the homomorphism~$F_* : H_n(X) \to H_n(X)$.
\end{defi}

If $X$ is a finite $T_0$ space and $F:X\multimap X$ is $\slsc$ with acyclic values, the definition of $L(F)$ depends on the $\susc$ map $F':X^{op}\multimap X$. On the other hand, $F^{op}:X^{op}\multimap X^{op}$ is also $\susc$ and $L(F^{op})$ is defined as well. The next result shows that in fact these two numbers are the same.

\begin{lema} \label{lslsc}
Let $X$ be a finite $T_0$ space and let $F:X\multimap X$ be a $\slsc$ multivalued map with acyclic values. Then $L(F)=L(F^{op})$.
\end{lema}
\begin{proof}
We consider $F^{op}$ as a subspace of $X^{op}\times X^{op}$ and call $p_1^{op}, p_2^{op}$ the projections onto the first and second coordinates. Similarly $F'$ is a subspace of $X^{op}\times X$ and we call $p_1',p_2'$ the projections onto $X^{op}$ and $X$ respectively. Define a multivalued map $\K \circ F: X^{op}\to \K(X)$ by $\K \circ F (x)=\K(F(x))$. Then $\K \circ F$ is $\susc$ with acyclic values. Once again $\K \circ F$ is considered as a subspace of $X^{op}\times \K(X)$ and by Lemma \ref{lem:inducedhommap} $p_1:\K \circ F \to X^{op}$ induces isomorphisms in homology. We have the following commutative diagram

\begin{displaymath}
\xymatrix@C=14pt{ & F^{op} \ar@{->}_{p_1^{op}}[dl] \ar@{->}^{p_2^{op}}[rr] & & X^{op} & \\
								X^{op} & & \K \circ F \ar@{->}_{p_1}[ll] \ar@{->}^{p_2}[rr] \ar@{->}_{1\times \mu_{X^{op}}}[ul] \ar@{->}^{1\times \mu_{X}}[dl] & & \K(X) \ar@{->}^{\mu_{X}}[dl] \ar@{->}_{\mu_{X^{op}}}[ul] \\
								& F' \ar@{->}^{p_1'}[ul] \ar@{->}_{p_2'}[rr] & & X & } 
\end{displaymath}

Note that the maps $1\times \mu_{X^{op}}:\K \circ F\to F^{op}$ and $1\times \mu_{X}:\K \circ F\to F'$ are well-defined. By definition $F^{op}_*=(p_2^{op})_*(p_1^{op})_*^{-1}: H_n(X^{op})\to H_n(X^{op})$ and $F_*=(p_2')_*(p_1')_*^{-1}(\mu_{X^{op}})_*(\mu_X)_*^{-1}$. By commutativity of the diagram $$F^{op}_*=(\mu_{X^{op}})_*(\mu_{X})_*^{-1}(p_2')_*(p_1')_*^{-1}=(\mu_{X^{op}})_*(\mu_{X})_*^{-1}F_*(\mu_{X})_*(\mu_{X^{op}})_*^{-1}.$$

Thus, $F^{op}_*$ and $F_*$ are conjugate so they have the same trace in each degree and then $L(F^{op})=L(F)$.

\end{proof}

We provide a couple of instructive examples.

\begin{ej} \label{ej1}
Let~$X$ be the finite space of four points $a$, $b$, $c$, and~$d$ depicted in Figure~\ref{example1}. This space is weak homotopy equivalent to~$S^1$, since its order complex is homeomorphic to~$S^1$. Consider the multivalued map $F : X \multimap X$ defined by
\begin{displaymath}
  F(a) = \{a,b,c\} \; , \quad
  F(b) = \{a,b,d\} \; , \quad
  F(c) = \{a\} \; , \quad\mbox{and}\quad
  F(d) = \{b\} \; .
\end{displaymath}
Then~$F$ is~$\susc$ with acyclic values, and as before we can identify~$F$ with its graph, which is a subset of~$X \times X$. One can easily see that the poset $F \subseteq X \times X$ has~$8$ points and is also weak homotopy equivalent to~$S^1$.
\begin{figure}[tb]
\begin{center}
\includegraphics[scale=0.6]{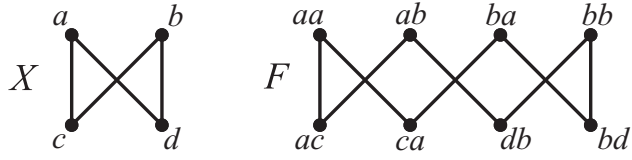}
\caption{The finite model $X$ of $S^1$ and $F\subseteq X\times X$, which are discussed in Example~\ref{ej1}.}\label{example1}
\end{center}
\end{figure}
Therefore, the maps induced in $H_1$ by $p_1: F \to X$ and $p_2: F \to X$ have $1x1$ matrices and a straightforward calculation
shows that with the same choice of generators of~$H_1(F)$ and~$H_1(X)$ for both maps the matrices are $[1]$  and $[-1]$ or vice versa.
%The map $p_1: F \to X$ has degree~$-1$ for a certain choice of the generators of~$H_1(F)$ and~$H_1(X)$, while $p_2 : F \to X$ has degree~$1$.
Therefore, the map $F_* : H_1(X) \to H_1(X)$ is given by~$-1_{\Z}$, and the Lefschetz number of~$F$ can be computed as $L(F) = 1-(-1) = 2$.
\end{ej}

\begin{ej}
Let~$X$ be a finite~$T_0$ space and let~$Y$ be any topological space. Let $f:Y\to X$ be a continuous map such that $f^{-1}(U_x)$ is acyclic for every $x\in X$. By Theorem~\ref{quillenhomology}  (McCord's Theorem for homology), the induced map on homology $f_* : H_n(Y)\to H_n(X)$ is an isomorphism for every $n\ge 0$. Let $F:X\multimap Y$ be the multivalued map defined by $F(x)=f^{-1}(U_x)$. Then~$F$ is $\susc$ with acyclic values. We observe that $F_* : H_n(X) \to H_n(Y)$ is the inverse of the map~$f_*$. If $p_1 : F \to X$ and $p_2 : F \to Y$ denote our earlier projections, then $f p_2 \le p_1$. Therefore, the composition~$f p_2$ is homotopic to~$p_1$. Thus, we obtain $f_* (p_2)_* = (p_1)_* : H_n(F) \to H_n(X)$, and then $f_* F_* = f_* (p_2)_* (p_1)_*^{-1} = 1_{H_n(X)}$. It follows that $F_*$ is the inverse of~$f_*$, as claimed.
\end{ej}

We close this section with a proof of the result mentioned above, that homomorphisms induced in homology by $\slsc$ maps between finite spaces can be defined in an alternative way. This will not be used in the following sections.

\begin{prop}
Let $X$, $Y$ be finite $T_0$ spaces and let $F:X\multimap Y$ be a $\slsc$ multivalued map with acyclic values. Then $p_1:F\to X$ induces isomorphisms in homology and $F_*:H_n(X)\to H_n(Y)$ coincides with the composition $(p_2)_*(p_1)^{-1}_*$.
\end{prop}
\begin{proof}
The proof is a refinement of the proof of Lemma \ref{lslsc}. Consider the multivalued map $\K \circ F:X^{op}\multimap \K(Y)$ defined by $\K \circ F (x)=\K(F(x))$. It is $\susc$ with acyclic values, so $p_1:\K \circ F \subseteq X^{op} \times \K(Y) \to X^{op}$ induces isomorphisms in homology. The maps $F':X^{op}\multimap Y$ and $F^{op}:X^{op}\multimap Y^{op}$ are also $\susc$ with acyclic values, so $p_1':F'\to X^{op}$ and $p_1^{op}:F^{op}\to X^{op}$ induce isomorphisms in homology as well. On the other hand, by applying the functor $\K$ to $X\leftarrow F \to Y$, we obtain maps $\K(p_1):\K(F)\to \K(X)$ and $\K(p_2):\K(F)\to \K(Y)$. There is a commutative diagram

\begin{displaymath}
\xymatrix@C=14pt{ & & & & & & F' \ar@{->}^{p_2'}[dddd] \ar@{-->}_{p_1'}[lllllldd] \\
									& & & & & \K \circ F \ar@{-->}_{p_1}[llllld] \ar@{-->}^{1\times \mu_{Y^{op}}}[llld] \ar@{->}^{p_2}[dd] \ar@{-->}_{1\times \mu_Y}[ru] & \\
									X^{op} & & F^{op} \ar@{-->}^{p_1^{op}}[ll] \ar@{->}_{p_2^{op}}[rr] & & Y^{op} & & \\
									& \K(X) \ar@{-->}^{\mu_{X^{op}}}[lu] \ar@{-->}_{\mu_X}[rd] & & \K(F) \ar@{-->}^{\mu_{F^{op}}}[lu] \ar@{-->}^{\K(p_1)}[ll] \ar@{->}_{\K(p_2)}[rr] \ar@{-->}_{\mu_F}[rd] & & \K(Y) \ar@{-->}_{\mu_Y}[rd] \ar@{-->}^{\mu_{Y^{op}}}[lu] & \\
									& & X & & F \ar@{-->}^{p_1}[ll] \ar@{->}_{p_2}[rr] & & Y } 
\end{displaymath} in which the dashed arrows represent maps that induce isomorphisms in homology. Since $p_1^{op}$ induces isomorphisms in homology, so does $\K(p_1)$ and then $p_1:F\to X$. By commutativity $F_*=(p_2')_*(p_1')^{-1}_*(\mu_{X^{op}})_*(\mu_X)_*^{-1}: H_n(X)\to H_n(Y)$ coincides with $(p_2)_*(p_1)^{-1}_*$.
\end{proof}

\section{A fixed point theorem}
\label{secfixedpointthm}

We now turn our attention to proving the Lefschetz fixed point theorem for multivalued maps on finite topological spaces. For this, suppose that~$X$ is a finite~$T_0$ space and that $F : X \multimap X$ is a multivalued map. Recall that a point $x \in X$ is said to be a fixed point of~$F$ if we have $x \in F(x)$. In our main theorem of this section, we will show that if~$F$ satisfies the assumptions of Definition~\ref{def:lefschetz} and if~$L(f) \neq 0$, then~$F$ has a fixed point. First, however, we need an auxiliary result.

Recall that an acyclic carrier from a simplicial complex~$K$ to another complex~$L$ is a function~$\Phi$ which assigns an acyclic subcomplex~$\Phi (\sigma) \subseteq L$ to each simplex $\sigma \in K$ in such a way that $\sigma \subseteq \sigma'$ implies $\Phi (\sigma ) \subseteq \Phi (\sigma')$. 
Note that by viewing the collection of simplices of a finite simplicial complex $K$ as a finite topological space $\X(K)$ with its topology induced by the face relation, we may interpret an acyclic carrier between finite complexes $K$ and $L$ as a susc multivalued map $\X(K)\multimap \X(L)$ with open and acyclic values. 

The Acyclic Carrier Theorem (\cite[Theorem 13.3]{Mun}) says that if~$\Phi$ is an acyclic carrier from~$K$ to~$L$, there exists a chain map $\varphi : C_*(K) \to C_*(L)$ such that for every oriented $n$-simplex $\sigma \in K$ the chain $\varphi (\sigma)$ lies in $C_n(\Phi (\sigma))$. Any such map is said to be carried by~$\Phi$. Two chain maps carried by the acyclic carrier~$\Phi$ are always chain homotopic. If we have $L=K$, then we can define the Lefschetz number~$L(\Phi)$ of~$\Phi$ as the Lefschetz number of any chain map $\varphi : C_*(K) \to C_*(K)$ carried by~$\Phi$. By the Hopf trace theorem the Lefschetz number of the chain map~$\varphi$ coincides with $\sum_{n \ge 0}(-1)^n \tr(\varphi_n)$, where as usual $\varphi_n : C_n(K) \to C_n(K)$.

\begin{lema}(Lefschetz fixed point theorem for acyclic carriers) \label{lac}
Let~$K$ be a finite simplicial complex and let~$\Phi$ be an acyclic carrier from~$K$ to itself. If $L(\Phi) \neq 0$, then there exists a simplex $\sigma \in K$ which is contained in $\Phi (\sigma)$.
\end{lema}
\begin{proof}
Let $\varphi:C_*(K)\to C_*(K)$ be a chain map which is carried by~$\Phi$. According to our hypothesis there exists at least one number $n\ge 0$ such that $\tr (\varphi_n)\neq 0$ and, in particular, there is an oriented $n$-simplex~$\sigma$ such that $\varphi(\sigma)\in C_n(K)$ is a chain which contains~$\sigma$ in its support. Since $\varphi (\sigma) \in C_n(\Phi(\sigma))$, we deduce that $\sigma \in \Phi(\sigma)$.
\end{proof}

For now, we restrict our attention to multivalued maps which are $\susc$. Thus, let~$X$ be a finite~$T_0$ space and let $F:X\multimap X$ be a $\susc$ multivalued map with acyclic values. We define an acyclic carrier~$\Phi_F$ from the order complex~$\K(X)$ to itself by setting
\begin{equation} \label{eqn:PhiF}
  \Phi_F(\sigma) = \K(F(\max{\sigma}))
  \quad\mbox{ for all }\quad
  \sigma \in \K(X) \; .
\end{equation}
Note that due to the acyclicity assumption, the map~$\Phi_F$ is indeed an acyclic carrier.

\begin{prop} \label{triangulo}
Let~$X$ be a finite~$T_0$ space and let $F:X\multimap X$ be a $\susc$ multivalued map with acyclic values. Then we have $L(F)=L(\Phi_F)$, if~$\Phi_F$ is defined as in~(\ref{eqn:PhiF}).
\end{prop}
\begin{proof}
Let $\varphi:C_*(\K(X))\to C_*(\K(X))$ be a chain map carried by $\Phi_F$. Once again we identify $F$ with a subspace of $X\times X$. We will show that the following triangle of chain maps commutes up to chain homotopy.
\begin{displaymath}
\xymatrix{ C_*(\K(F)) \ar^{\K(p_2)_*}[r] \ar_{\K(p_1)_*}[d] & C_*(\K(X)) \\
C_*(\K(X)) \ar_{\varphi}[ur] &  }
\end{displaymath}

\medskip

Define an acyclic carrier $\Lambda : \K(F) \to \K(X)$ as follows. Let~$\sigma$ be a simplex of~$\K(F)$, i.e., a chain $(x_0,y_0)<(x_1,y_1)<\ldots <(x_n,y_n)$ such that $y_i\in F(x_i)$ for $i=0,\ldots,n$. Then we define $\Lambda(\sigma)=\K(F(x_n))$. Since we assumed that the multivalued map~$F$ is $\susc$ with acyclic values, the map~$\Lambda$ is an acyclic carrier.

Suppose now that $\sigma=[(x_0,y_0),(x_1,y_1),\ldots,(x_n,y_n)]$ is an oriented simplex of~$\K(F)$, where $(x_0,y_0)<(x_1,y_1)<\ldots <(x_n,y_n)$. Note that the chain $\K(p_2)_*(\sigma)$ is the oriented simplex $[y_0,y_1,\ldots ,y_n]$ if all the~$y_i$ are different, and~$0$ if two of them are equal. Since~$F$ is $\susc$, we immediately obtain $y_i\in F(x_i) \subseteq F(x_n)$ for every $i = 0,\ldots,n$, and therefore the inclusion $\K(p_2)_*(\sigma) \in C_*(\K(F(x_n))) = C_*(\Lambda(\sigma))$ is satisfied. This implies that~$\K(p_2)_*$ is carried by~$\Lambda$. On the other hand, the chain~$\K(p_1)_*(\sigma)$ is either equal to $[x_0,x_1,\ldots,x_n]$, or it is~$0$. Thus, we have $\varphi \K(p_1)_*(\sigma)\in C_*(\K(F(x_n)))$. Hence, $\varphi \K(p_1)_*$ is also carried by~$\Lambda$. On the level of homology the acyclic carrier theorem then implies the equality
\begin{displaymath}
  \varphi_* \K(p_1)_* = \K(p_2)_* :
  H_n(\K(F)) \to H_n(\K(X))
  \quad\mbox{ for every }\quad
  n\ge 0 \; .
\end{displaymath}
This in turn yields the identity $\varphi_*=\K(p_2)_*\K(p_1)_*^{-1}$. By McCord's theorem, $\K(p_2)_*\K(p_1)_*^{-1}$ and $(p_2)_*(p_1)_*^{-1}:H_n(X)\to H_n(X)$ merely differ in a conjugation by an isomorphism. This finally implies the equality $L(\Phi_F)=L(\varphi)=L(F)$.
\end{proof}

After these preparations we can now deduce the Lefschetz fixed point theorem for strongly semicontinuous multivalued maps with acyclic values between finite~$T_0$ spaces.

\begin{teo}[Lefschetz fixed point theorem] \label{main}
Let~$X$ be an arbitrary finite~$T_0$ space and let $F:X\multimap X$ be a multivalued map with acyclic values which is $\susc$ or $\slsc$. If the inequality $L(F) \neq 0$ holds, then~$F$ has a fixed point.
\end{teo}
\begin{proof}
We first assume that~$F$ is $\susc$. By the previous proposition, one obtains $L(\Phi_F)\neq 0$, and by Lemma~\ref{lac}, there exists a simplex $\sigma \in \K(X)$ such that $\sigma \in \Phi_F(\sigma) = \K(F(\max (\sigma)))$. This furnishes in particular the inclusion $\max (\sigma)\in F(\max (\sigma))$.

If on the other hand~$F$ is $\slsc$, by Lemma \ref{lslsc} $L(F^{op})=L(F)\neq 0$ and by the case already proved $F^{op}:X^{op}\multimap X^{op}$ has a fixed point, and so does $F$.
\end{proof}

If $f:X \to X$ is a continuous single-valued map on a finite~$T_0$ space, then one can define an associated multivalued map $F_f:X\multimap X$ by letting $F_f(x) = U_{f(x)}$. This map is clearly $\susc$. Moreover, since for each $x \in X$ the image~$F_f(x)$ has a maximum, the map~$F_f$ has acyclic values. The Lefschetz numbers of~$f$ and~$F_f$ coincide since the chain map $\K(f)_* : C_*(\K(X)) \to C_*(\K(X))$ is carried by~$\Phi_{F_f}$. Moreover, if~$F_f$ has a fixed point, say $x \in X$, then $x \in U_{f(x)}$, so $x \lee f(x)$. In this case it is easy to see that~$f$ also has to have a fixed point. Since the map~$f$ is order-preserving we have a chain $x \lee f(x) \lee f^2(x) \lee \ldots$, and by the finiteness of~$X$ there exists an $n \ge 0$ such that~$f^n(x)$ is a fixed point of~$f$. In particular, Theorem~\ref{main} implies that any continuous single-valued map with non-trivial Lefschetz number has a fixed point. Baclawski and Bj\"orner show in~\cite[Theorem 1.1]{BB} that in fact~$L(f)$ is the Euler characteristic~$\chi (X^f)$ of the fixed point set. The analogue for multivalued maps is not true as the next example shows.

\begin{ej}
Let~$X$ be the finite model of~$S^1$ which has already been considered in Figure~\ref{example1}. Let $F:X\multimap X$ be defined by $F(a) = F(b) = \{a,b,c\}$, $F(c) = \{a\}$, as well as $F(d) = \{b\}$. This multivalued map is $\susc$ with acyclic values. The subspace~$X^F$ of fixed points is the discrete space of two points $\{a,b\}$, and therefore we have $\chi(X^F) = 2$. On the other hand, the image of~$F$ is the contractible space~$\{a,b,c\}$. Therefore, $p_2 : F \to X$ is null-homotopic, which in turn implies~$L(F)=1$.
\end{ej}

A selector of a multivalued map $F : X \multimap Y$ is a single-valued map $f : X \to Y$ such that $f(x) \in F(x)$ for every $x \in X$. We would like to point out that a given $\susc$ multivalued map $F : X \multimap X$ with acyclic values may not have a continuous selector, even if each value~$F(x)$ is a contractible finite space. This is the case for the map studied in Example~\ref{ej1}. However, the following result holds for $\susc$ maps, and it can easily be adapted to the $\slsc$ case as well.

\begin{prop}[Order complex selector] \label{induced}
If~$X$ is a finite~$T_0$ space and $F : X \multimap X$ is a $\susc$ multivalued map such that~$F(x)$ is weakly contractible for every $x \in X$, then there exists a continuous map $f : \K(X) \to \K(X)$ such that for every simplex $\sigma \in \K (X)$ we have $f(\sigma) \subseteq \K(F(\max (\sigma)))$. The Lefschetz number of any such map~$f$ is~$L(F)$ and, furthermore, if~$f$ has a fixed point, then so does~$F$.
\end{prop}
\begin{proof}
The existence of~$f$ is guaranteed by Walker's contractible carrier theorem~\cite[Lemma 2.1]{Wal}, since~$\K(F(x))$ is contractible for every $x \in X$. In addition, if~$f$ has a fixed point~$\alpha$ which lies in an open simplex $\sigma \in \K(X)$, then the closed simplex~$\sigma$ is contained in the subcomplex $\K(F(\max (\sigma)))$. In particular $\max (\sigma)$ is a fixed point of~$F$.

It remains to show that the Lefschetz number of~$f$ equals the Lefschetz number of~$F$. For this, let $\psi:\K(X)'\to \K(X)$ be a simplicial approximation of~$f$, where~$\K(X)'$ is a subdivision of~$\K(X)$. The maps $f_*:H_n(\K(X))\to H_n(\K(X))$ are the homomorphisms induced by the chain map $\psi_* \lambda  : C_*(\K(X))\to C_*(\K(X))$, where $\psi_*:C_*(\K(X)')\to C_*(\K(X))$ is the chain map induced by~$\psi$ and $\lambda :C_*(\K(X))\to C_*(\K(X)')$ is the subdivision operator~\cite[Theorem~17.2]{Mun}. Let $\rho:\K(X)'\to \K(X)$ be a simplicial approximation to the identity and let $\varphi : C_*(\K(X))\to C_*(\K(X))$ be a chain map carried by~$\Phi_F$. Define one last acyclic carrier~$\Theta$ from~$\K(X)'$ to~$\K(X)$ by $\Theta(\tau) = \K(F(\max(\sigma)))$, where $\sigma \in \K(X)$ is the simplex of smallest dimension containing~$\tau$. It is clear that~$\Theta$ is an acyclic carrier. Both~$\psi_*$ and~$\varphi_*\rho_*$ are carried by this acyclic carrier~$\Theta$. Indeed, if $\tau \in \K(X)'$ and~$\alpha$ is a point in the open simplex~$\overset{\circ}{\tau}$, then $\alpha \in \overset{\circ}{\sigma}$, since by construction $\sigma \in \K(X)$ denotes the smallest simplex containing~$\tau$. This yields the inclusion $f(\alpha) \in \K(F(\max(\sigma)))$. Since~$\psi$ is an approximation of~$f$, the image~$\psi (\alpha)$ lies in the same subcomplex, and therefore we obtain the inclusion $\psi(\tau) \in \K(F(\max(\sigma))) = \Theta(\tau)$. Thus, the map~$\psi_*$ is carried by~$\Theta$. On the other hand, the inclusion $\rho (\tau) \subseteq \sigma$ holds, and therefore $\varphi_* \rho_*(\tau)\in C_*(\Phi_F(\sigma)) = C_*(\Theta(\tau))$. According to the acyclic carrier theorem we have $\psi_* \simeq \varphi_* \rho_*$, which in turn implies $\psi_*\lambda\simeq \varphi_*\rho_*\lambda\simeq \varphi_*$. This proves that on the level of homology one has $f_* = F_*:H_n(\K(X))\to H_n(\K(X))$, and establishes in particular the equality $L(f) = L(F)$.
\end{proof}

For any continuous map $f:X\to X$ on a finite~$T_0$ space~$X$ we have earlier defined the multivalued map $F_f:X\multimap X$ which maps $x$ to $U_{f(x)}$. Moreover, we have shown that if~$F_f$ has a fixed point, then so does~$f$. Of course, the map~$f$ is a selector of the map~$F$ in this case. In general, however, it is not true that if a $\susc$ multivalued map $F:X\multimap X$ with acyclic values has a fixed point, then every continuous selector has a fixed point. This is illustrated in the following example.

\begin{ej} \label{corona}
Let~$X$ denote the $6$-point space depicted in Figure~\ref{example2}. Consider the multivalued map $F : X\multimap X$ defined by $F(a')=\{b'\}$, $F(b')=\{c'\}$, $F(c')=\{a'\}$, $F(b)=\{b',c,c'\}$, $F(c)=\{c',a,a'\}$, as well as $F(a)=\{c',a,a',b,b'\}$. One can easily verify that the map~$F$ is $\susc$ with acyclic values, and that the point~$a$ is a fixed point.
\begin{figure}[tb]
\begin{center}
\includegraphics[scale=0.6]{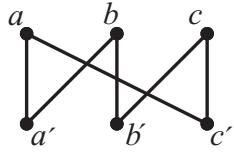}
\caption{A space with a multivalued map with a fixed point, and which has a selector without fixed points. For more details, see Example~\ref{corona}.}\label{example2}
\end{center}
\end{figure}
There exists a unique continuous selector $f:X\to X$, $a\mapsto b$, $a' \mapsto b'$, $b \mapsto c$, $b' \mapsto c'$, $c \mapsto a$, $c' \mapsto a'$, which is in fact fixed point free.
\end{ej}

\section{The fixed point property for multivalued maps}
\label{secfpp}

A topological space~$X$ is said to have the fixed point property (FPP) if every continuous self map $X\to X$ has a fixed point. By the Lefschetz fixed point theorem for finite spaces~\cite[Theorem~1.1]{BB}, if a connected finite~$T_0$ space has trivial rational homology groups~$H_n(X;\mathbb{Q})$ for every $n\ge 1$ (in which case it is called a rationally acyclic space), then it has the FPP. The FPP is a homotopy invariant of finite~$T_0$ spaces, but it is not a weak homotopy invariant. Bj\"orner and Baclawski found in~\cite{BB} examples of finite spaces which are weakly homotopy equivalent to~$S^2$ and which have the FPP. Later, in~\cite{Bar3}, it was proved that for any compact CW-complex there exists a finite space which is weakly homotopy equivalent and which has the FPP.

We will say that a finite~$T_0$ space has the \textit{fixed point property for multivalued maps} (MFPP) if every $\susc$ multivalued map $X\multimap X$ with acyclic values has a fixed point. It is clear that any finite~$T_0$ space~$X$ with the MFPP has the FPP for if $f:X\to X$ is a continuous map, then the multivalued map $F_f:X\multimap X$ defined in the last section has a fixed point, and this in turn implies that so does~$f$. On the other hand Theorem~\ref{main} implies that any rationally acyclic finite~$T_0$ space has the MFPP since any multivalued map $X\multimap X$ with acyclic values has Lefschetz number equal to~1. Therefore we have the following two implications

$$ \textrm{rationally acyclic } \Rightarrow MFPP \Rightarrow FPP.$$

\medskip

We will now show that both of these implications are in fact strict.

Recall that a beat point in a finite~$T_0$ space~$X$ is a point $x\in X$ which covers a unique element or it is covered by a unique element, i.e., either the set~$X_{< x}:=\{x'\in X: x'<x\}$ has a maximum or~$X_{> x}:=\{x'\in X: x'>x\}$ has a minimum. If~$x$ is a beat point of~$X$, then one can show that~$X$ and~$X\smallsetminus \{x\}$ are homotopy equivalent. In particular, if we can remove beat points one by one to obtain a singleton, then the original space is contractible. In fact the converse holds: If a finite~$T_0$ space is contractible, it is possible to remove beat points one by one to obtain just one point (\cite{Bar2,Sto}).

\begin{ej} \label{esfera}

Let~$X$ be the finite~$T_0$ space with Hasse diagram depicted in Figure~\ref{sphere}. The order complex~$\K (X)$ of~$X$ is homeomorphic to the $2$-dimensional sphere. This space has the FPP. Indeed, if a continuous map $f:X\to X$ is a homeomorphism, then it fixes the point~$a$ --- the unique maximal point which covers two elements. If  $f$ is not a homeomorphism, then it is not surjective. Hence, $\K (f):S^2\to S^2$ is null-homotopic and $L(f) = L(\K(f)) = 1$. This example is very similar to Example~2.4 in~\cite{BB}.
\begin{figure}[tb]
\begin{center}
\includegraphics[scale=0.6]{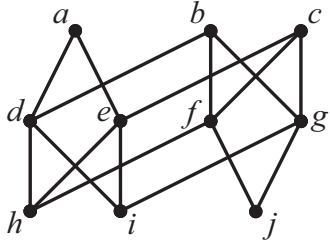}
\caption{A space with the FPP which does not have the MFPP. For more details see Example~\ref{esfera}.}\label{sphere}
\end{center}
\end{figure}

We now define a fixed point free multivalued $\susc$ map $F:X\multimap X$ with acyclic values. For this, let $F(x) = (X_{\ge x})^c = \{y\in X \ | \ y \ngeq x\}$. This map is clearly $\susc$ and, in this particular case, it is easy to check that each value~$F(x)$ is contractible, since beat points can be removed one by one to obtain a singleton.
Obviously, $F$ has no fixed point. Thus, $X$ has the FPP but not the MFPP.
\end{ej}

Recall that the standard complex~$K_{\PP}$ of a presentation $\PP = \langle x_1,x_2,\ldots , x_n | r_1,r_2,\ldots ,r_m\rangle$ of a group~$G$ is the $2$-dimensional CW-complex which has a unique $0$-dimensional cell~$e^0$, a $1$-dimensional cell~$e_i^1$ for each generator~$x_i$ and a $2$-dimensional cell~$e_j^2$ for each relator~$r_j$. The attaching map of~$e_j^2$ follows the $1$-cells corresponding to the letters which appear in~$r_j$ with the orientation given by the exponent of each letter. The fundamental group of~$K_{\PP}$ is isomorphic to~$G$. A subpresentation of~$\PP$ is a presentation~$\mathcal{Q}$ whose generators are generators of~$\PP$ and whose relators are relators of~$\PP$. If~$\mathcal{Q}$ is a subpresentation of~$\PP$, then~$K_{\mathcal{Q}}$ is a subcomplex of~$K_{\PP}$. The trivial presentation is the presentation $\langle | \rangle$ of the trivial group with no generators and no relators. We identify a presentation with its standard complex, so we will say that a presentation~$\PP$ is contractible if~$K_{\PP}$ is contractible, and so on.

If~$K$ is a simplicial complex, then~$\X(K)$ denotes the face poset of~$K$, i.e.,  the poset of simplices of~$K$ ordered by inclusion. Note that~$\K(\X(K))$ is the barycentric subdivision~$K'$ of~$K$.

\begin{prop} \label{pres}
Let $\PP=\langle x_1,x_2,\ldots , x_n | r_1,r_2,\ldots ,r_m\rangle$ be a presentation of a group such that the standard complex~$K_{\PP}$ has the FPP and no nontrivial subpresentation of~$\PP$ is acyclic. Then for any triangulation~$K$ of~$K_{\PP}$, the finite space~$\X(K)$ has the MFPP.
\end{prop}
\begin{proof}
Let~$K$ be a triangulation of~$K_{\PP}$ and let $F:\X(K)\multimap \X(K)$ be a $\susc$ multivalued map with acyclic values. Then, for each $x\in \X(K)$, the image $F(x)\subseteq \X(K)$ is acyclic and weakly homotopy equivalent to~$\K (F(x))$, which is a subcomplex of the barycentric subdivision~$K'$ of~$K$. If we show that every acyclic subcomplex of~$K'$ is contractible, then~$F(x)$ is weakly contractible for every $x\in \X(K)$ and by Proposition~\ref{induced} there is an induced map $f:K'\to K'$ which has a fixed point by hypothesis, and then~$F$ also has a fixed point.

Let~$K$ be any triangulation of~$K_{\PP}$. We will prove that every acyclic subcomplex of~$K$ is contractible (we do not need~$K$ to be a barycentric subdivision). Let~$L$ be an acyclic subcomplex of~$K$. By performing simplicial collapses we may assume~$L$ has no free faces. Let $e^2_1, e^2_2,\ldots, e^2_m$ denote the open $2$-dimensional cells of~$K_{\PP}$, where~$e^2_j$ corresponds to the relator~$r_j$. Suppose that~$\sigma$ is a $2$-dimensional simplex of~$L$ whose interior contains a point~$x$ of~$e^2_j$ for some $1\le j\le m$. Let~$\tau$ be a $2$-dimensional simplex of~$K$ whose interior contains a point~$y$ of~$e^2_j$. There exists a simple path $\gamma:[0,1]\to K$ from~$x$ to~$y$ entirely contained in~$e^2_j$ which does not pass through any vertex of~$K$. If $\tau \notin L$, define $L_c\leqslant K$ to be the subcomplex of~$K$ generated by the $2$-simplices of~$K$ which are not in~$L$. Let~$t_0$ be the minimal $t\in [0,1]$ for which $\gamma(t)\in L_c$. Then~$\gamma(t_0)$ lies in an open $1$-simplex of~$L$ which is a free face of~$L$, contradicting our assumption. This proves that all of~$e^2_j$ is contained in~$L$. Thus, if a closed $2$-simplex of~$L$ intersects the open cell~$e^2_j$, then $e^2_j\subseteq L$. We deduce that the subcomplex $L_2\leqslant L$ generated by the $2$-simplices of~$L$ is a union of closed $2$-cells~$\overline{e}^2_j$. Then~$L_2$ is the standard subcomplex of a nontrivial subpresentation of~$\PP$ or $L_2=\emptyset$. In the first case~$L_2$ is connected and not acyclic by hypothesis. Since~$L$ is obtained from~$L_2$ by adding $0$-simplices and $1$-simplices, the complex~$L$ is not acyclic either, a contradiction. Therefore, $L_2=\emptyset$, so~$L$ is a $1$-dimensional acyclic complex and hence contractible.
\end{proof}

Standard complexes of presentations with the FPP were studied by Sadofschi Costa in~\cite{SC}. This leads to the following example.

\begin{ej}
Consider the presentation $$\PP=\langle x,y | x^3, xyx^{-1}yxy^{-1}x^{-1}y^{-1}, x^{-1}y^{-4}x^{-1}y^2x^{-1}y^{-1}\rangle$$ given in~\cite[Corollary~2.7]{SC}. The standard complex~$K_{\PP}$ has the FPP. It is not rationally acyclic, since its Euler characteristic is~$\chi(K_{\PP})=2$. There are~$11$ non-trivial subpresentations $\mathcal{Q}_1, \mathcal{Q}_2, \ldots, \mathcal{Q}_{11}$. Their (integral) homology groups~$H_1(K_{Q_i};\Z)$ of degree~$1$ are $\Z_3\oplus \Z_3$, $\Z_3\oplus \Z$, $\Z \oplus \Z$, $\Z_3$ or $\Z$ in all the cases. Due to Proposition~\ref{pres}, for any triangulation~$K$ of~$K_{\PP}$ the finite space~$\X(K)$ is not rationally acyclic, but it does have the MFPP.
\end{ej}

A finite~$T_0$ space is \textit{minimal} if it has no beat points.
%Every finite topological space~$X$ is homotopy equivalent, up to homeomorphism, to a unique minimal finite space, called the \textit{core} of~$X$.
Every finite topological space~$X$ is homotopy equivalent to a  minimal finite space, called the \textit{core} of~$X$.
The core of~$X$ is unique up to a homeomorphism.
It is always a retract of~$X$.

\begin{obs}
If a finite~$T_0$ space~$X$ is minimal and does not have the MFPP, then any finite~$T_0$ space~$Y$ which is homotopy equivalent to~$X$ also lacks the MFPP. To show this, suppose that $F:X\multimap X$ is a fixed point free $\susc$ multivalued map with acyclic values. Then the composition $iFr:Y\multimap Y$ is fixed point free with the same properties. Here, $i:X\to Y$ and $r:Y\to X$ are continuous maps which satisfy $ri=1_X$, and~$iFr$ is the composition, defined by $iFr(y)=i(F(r(y)))$. This remark applies for instance to the space~$X$ in Example~\ref{esfera}.
\end{obs}

We do not know whether the MFPP is a homotopy invariant. Other connections among fixed points, Lefschetz numbers and homotopies are discussed in Section~\ref{sechomotopy}.

\section{A different class of multivalued maps}
\label{secisotone}

During the last three sections of this paper we have focused on multivalued maps between finite~$T_0$ spaces which have acyclic values and which are either $\susc$ or $\slsc$. For such maps we could construct induced homomorphisms in homology, and prove a Lefschetz fixed point theorem. These results are inspired by the classical version of this theorem for acyclic maps of an ANR~$X$, which can be found for example in~\cite[Theorem~32.9]{Gor}. This result states that if~$X$ is a compact ANR and if $F:X\multimap X$ is an usc map with acyclic values, then the Lefschetz number $L(F)\in \Z$ is well-defined. Furthermore, if in this situation the Lefschetz number is nonzero, then the multivalued map~$F$ has a fixed point.

It is natural to wonder whether in the setting of finite~$T_0$ spaces the assumption of strong semicontinuity can be relaxed, or even just be replaced by another continuity assumption. Recall that a multivalued map $F:X\to Y$ is continuous in the sense of Michael if it is usc and lsc. Continuous maps between finite $T_0$ spaces where studied by Walker~\cite{Wal2}. He calls these maps isotone relations. There is no mention in~\cite{Wal2} of a version of the Lefschetz fixed point theorem for isotone maps between finite~$T_0$ spaces. There is, though, a characterization of those finite~$T_0$ spaces with the fixed point property with respect to isotone maps: they are exactly the contractible spaces. The class of finite spaces with the MFPP seems harder to be characterized. This class strictly contains the contractible and, moreover, the rationally acyclic spaces, and it is in turn strictly contained in the class of spaces with the FPP.

We saw earlier in this paper that if~$X$ is a finite~$T_0$ space and $F:X\multimap X$ is a $\susc$ multivalued map with acyclic values, then for every continuous selector $f : X \to X$ we must have $L(f) = L(F)$, since~$\K(f)_*$ is carried by~$\Phi_F$. If we want to define Lefschetz numbers for a bigger class of maps, it is natural to require that every continuous selector $f$ of a map $F$ in this class satisfies $L(f)=L(F)$. This property holds in the classical context of usc maps between compact ANR. Recall the notion of homotopy between acyclic maps from~\cite[Definition~32.5]{Gor}. Two multivalued usc maps $F,G : X \multimap Y$ with acyclic values are \textit{homotopic}, if there exists an usc multivalued map $H : X \times [0,1] \multimap Y$ with acyclic values such that $H(x,0) = F(x)$ and $H(x,1) = G(x)$ for every $x\in X$. In Corollary~32.7 of~\cite{Gor} it is shown that for any compact ANR~$X$, homotopic maps $F,G:X\multimap X$ of this kind have the same Lefschetz number. If~$X$ is any space and $F:X\multimap X$ is usc with acyclic values, then any continuous selector $f:X\to X$ of~$F$ is homotopic to~$F$. Indeed, define $H:X\times [0,1]\multimap Y$ by $H(x,t)=\{f(x)\}$ for $t<1$ and $H(x,t)=F(x)$ for $t=1$, for every $x\in X$. Then~$H$ has acyclic values and it is usc, since for any open subset $U\subseteq Y$ the small preimage of~$U$ is given by $H^{-1}(U) = (f^{-1}(U)\times [0,1)) \cup (F^{-1}(U)\times \{1\})$. The latter set is open since $F^{-1}(U) \subseteq f^{-1}(U)$. Thus, a selector has the same Lefschetz number as the multivalued map, also in this context.

Our next result shows that there is no way to define Lefschetz numbers of continuous multivalued maps (isotone relations) between finite~$T_0$ spaces if we want the property above to hold.

\begin{prop} \label{noway}
There is no map~$\lambda$ which assigns an integer number~$\lambda(F)$ to every continuous multivalued map with acyclic values $F:X\multimap X$ of finite~$T_0$ spaces, and which satisfies the following properties:
\begin{enumerate}
\item $\lambda(f) = L(f)$ for each single-valued map,
\item $\lambda(f) = \lambda(F)$ for each selector~$f$ of~$F$.
\end{enumerate}
\end{prop}
\begin{proof}
Consider the space~$X$ depicted in Figure~\ref{example1} and define the map $F:X\multimap X$ by setting $F(a)=F(b)=\{a,b,c\}$ and $F(c)=F(d)=\{a,c,d\}$. Then one can easily verify that~$F$ is both usc and lsc, and that it has acyclic values. The identity~$1_X$ is a selector of~$F$ with Lefschetz number $L(1_X)$ equal to the Euler characteristic of~$X$, which is $0$. On the other hand, the map $f:X\to X$ given by $f(a)=f(b)=a$ and $f(c)=f(d)=c$ is another continuous selector which is null-homotopic, and therefore $L(f)=1$. Thus, we cannot have $L(1_X) = \lambda (F) = L(f)$.
\end{proof}

\section{Homotopies}
\label{sechomotopy}

In this final section we discuss homotopies between multivalued maps. For the sake of presentation, we only consider multivalued maps which are $\susc$. Nevertheless, the results can easily be  adjusted for maps which are $\slsc$.

Let~$X$ again denote a finite~$T_0$ space and let~$Y$ be any topological space. Two $\susc$ multivalued maps $F,G:X\multimap Y$ with acyclic values are said to be \textit{homotopic}, if there exists a $\susc$ multivalued map $H:X\times [0,1]\multimap Y$ with acyclic values which satisfies $H(x,0)=F(x)$ and $H(x,1)=G(x)$ for every $x\in X$.  Furthermore, for multivalued maps $F,G:X\multimap Y$ we write $F \le G$ if $F(x)\subseteq G(x)$ for every $x\in X$. Then the following result holds.

\begin{prop}[Homotopic strongly semicontinuous maps] \label{fence}
Let~$X$ and~$Y$ be two finite~$T_0$ spaces. Then two $\susc$ multivalued maps $F,G:X\multimap Y$ with acyclic values are homotopic, if and only if there exists a sequence, also called a \textit{fence}, of multivalued maps $F_i:X\multimap Y$ which are $\susc$ and have acyclic values, and which satisfy $F=F_0\le F_1\ge F_2\le \ldots F_n=G$.
\end{prop}
\begin{proof}
The proof is an adaptation of Stong's arguments in~\cite[Corollary~2]{Sto} to our context. A $\susc$ multivalued map $H:X\times [0,1]\multimap Y$ with acyclic values is a continuous single-valued map $X\times [0,1]\to \mathcal{P}(Y)$ with acyclic values by Remark \ref{susciffc}. By the exponential law \cite[Lemma~1]{Sto} this corresponds to a continuous path $\gamma :[0,1]\to \mathcal{P}(Y)^X$ in which $\gamma _t :X\to \mathcal{P}(Y)$ has acyclic values for every $0\le t\le 1$. There exists one such path from~$F$ to~$G$ if and only if~$F$ and~$G$ lie in the same path-connected component of the subspace~$S$ of~$\mathcal{P}(Y)^X$ given by those maps with acyclic values. But the compact-open topology in~$\mathcal{P}(Y)^X$ corresponds to the order~$\le$ for multivalued maps defined above, and therefore path-components of~$S$ are described by fences in~$S$.
\end{proof}

\begin{coro} \label{homot}
If~$X$ is a finite~$T_0$ space and $F,G:X\multimap X$ are homotopic $\susc$ multivalued maps with acyclic values, then $F_*=G_*:H_n(X)\to H_n(X)$ for every $n\ge 0$. In particular, we have $L(F) = L(G)$.
\end{coro}
\begin{proof}
By Proposition~\ref{fence} we may assume that $F\le G$. Let $\varphi, \gamma: C_*(\K(X))\to C_*(\K(X))$ be chain maps carried by~$\Phi_F$ and~$\Phi_G$, respectively. Since $F\le G$, the chain map~$\varphi$ is also carried by~$\Phi_G$, and therefore we have $\varphi_*=\gamma_*:H_n(X)\to H_n(X)$ for every $n$. By the proof of Proposition~\ref{triangulo} the definitions of~$F_*$ and~$G_*$ coincide.
\end{proof}

If $F,G:X\multimap Y$ are homotopic $\susc$ multivalued maps with acyclic values between not necessarily equal finite~$T_0$ spaces, then~$F$ and~$G$ induce the same homomorphisms in homology. The proof is identical to the proof of Corollary \ref{homot} using a straightforward modification of the proof of Proposition \ref{triangulo}.

\begin{ej}
If $f\ge g:X\to X$ are comparable continuous maps between finite~$T_0$ spaces, then~$f$ has a fixed point if and only if~$g$ does. The analogue for multivalued maps is not true. Let $F:X\multimap X$ be the map in Example~\ref{corona} and let $G:X\multimap X$ be the map which coincides with~$F$ in each point but~$a$, and for which $G(a)=\{a',b,b'\}$. Then both~$F$ and~$G$ are $\susc$ with acyclic values, we have $F\ge G$, and the map~$F$ has a fixed point --- but the map~$G$ is fixed point free. Of course, $L(F) = L(G)$ by the previous result, and this number therefore has to be~$0$ by the Lefschetz fixed point theorem.
\end{ej}

In the classical fixed point theory of multivalued maps many classes of maps $F:X\multimap X$ admit a homotopic single-valued continuous function $f:X\to X$. However, in the finite space setting, the map $F:X\multimap X$ in Example~\ref{ej1} is not homotopic to any single-valued map in the sense that there is no continuous map $f:X\to X$ such that $F_f:X\multimap X$ is homotopic to~$F$. In fact, if $G:X\multimap X$ is $\susc$ with acyclic values and satisfies $G \le F$, then $G(c)=\{a\}$, $G(d)=\{b\}$, $\{a,b\}\subseteq G(a) \subseteq \{a,b,c\}$ is acyclic, so $G(a)=\{a,b,c\}$ and similarly $G(b)=\{a,b,d\}$. It is also easy to prove that if $G \ge F$ then $G=F$, so by Proposition~\ref{fence}, the homotopy class of~$F$ contains exactly one map.

\end{document}